\documentclass[paper=a4, fontsize=11pt]{scrartcl} 

\usepackage[T1]{fontenc} 
\usepackage{fourier} 
\usepackage[english]{babel} 
\usepackage{amsmath,amsfonts,amsthm} 
\usepackage{amsmath, calligra, mathrsfs}
\usepackage{amssymb}
\usepackage[mathscr]{euscript}
\usepackage{verbatim}
\usepackage[dvipsnames]{xcolor}
\usepackage{mdframed} 
\usepackage{soulutf8}
\usepackage{stmaryrd}
\usepackage{tikz-cd}
\usepackage{hyperref}
\usepackage{lipsum} 

\usepackage{sectsty} 
\allsectionsfont{\centering \normalfont\scshape} 

\usepackage{fancyhdr} 
\usepackage{xurl}
\newcommand*{\sheafhom}{\mathcal{H}\kern -.5pt om}
\numberwithin{equation}{section} 
\numberwithin{figure}{section} 
\numberwithin{table}{section} 

\newtheorem{thm}{Theorem}[section]
\newtheorem{cor}[thm]{Corollary}
\newtheorem{prop}[thm]{Proposition}

\newtheorem{quest}[thm]{Question}
\theoremstyle{definition}
\newtheorem{defn}[thm]{Definition}

\newtheorem{exmp}[thm]{Example}

\theoremstyle{remark}
\newtheorem{rem}[thm]{Remark}

\DeclareMathOperator{\Ind}{Ind}

\DeclareMathOperator{\Pic}{Pic}

\DeclareMathOperator{\rk}{rank}

\newcommand{\horrule}[1]{\rule{\linewidth}{#1}} 

\title{	
	\normalfont \normalsize 
	\textsc{} \\ [25pt] 
	\horrule{0.5pt} \\[0.4cm] 
	\huge Symmetries and intrinsic vs. extrinsic properties of $\overline{\mathcal{M}}_{0, n}$ 
	
	\horrule{2pt} \\[0.5cm] 
}

\author{Soohyun Park} 

\date{\normalsize September 1, 2023} 

\begin{document}
	
	\maketitle 
	
	\begin{abstract}
		\noindent We consider the following question: How much of the combinatorial structure determining properties of $\overline{\mathcal{M}_{0, n}}$ is ``intrinsic'' and how much new information do we obtain from using properties specific to this space? Our approach is to study the effect of the $S_n$-action. Apart from being a natural action to consider, it is known that this action does not extend to other wonderful compactifications associated to the $A_{n - 2}$ hyperplane arrangement. We find the differences in intersection patterns of faces on associahedra and permutohedra which characterize the failure to extend to other compactifications and show that this is reflected by most terms of degree $\ge 2$ of the cohomology/Chow ring. \\
		
		\noindent Even from a combinatorial perspective, terms of degree 1 are more naturally related to geometric properties. In particular, imposing $S_n$-invariance implies that many of the log concave sequences obtained from degree 1 Hodge--Riemann relations (and all of them for $n \le 2000$) on the Chow ring of $\overline{\mathcal{M}_{0, n}}$ can be restricted to those with a special recursive structure. A conjectural result implies that this is true for all $n$. Elements of these sequences can be expressed as polynomials in quantum Littlewood--Richardson coefficients multiplied by terms such as partition components, factorials, and multinomial coefficients. After dividing by binomial coefficients, polynomials with these numbers as coefficients can be interepreted in terms of volumes or resultants. Finally, we find a connection between the geometry of $\overline{\mathcal{M}_{0, n}}$ and higher degree Hodge--Riemann relations of other rings via Toeplitz matrices.

	\end{abstract}
		
	\section{Introduction}
	
	The fact that the moduli space $\overline{\mathcal{M}_{0, n}}$ of stable rational curves with $n$ marked points has a concrete description as a wonderful compactification has been used in many results at the interface of algebraic geometry and combinatorics. This can be used to study properties of $\overline{\mathcal{M}_{0, n}}$ (e.g. the Euler characteristic) using purely combinatorial methods. While it is interesting that this is possible, one is naturally led to ask the following question: \\
	
	\begin{quest}
		How much of the combinatorial structure is ``intrinsic'' to the geometry of $\overline{\mathcal{M}_{0, n}}$? If there are interesting uniqueness properties, how much of this is reflected in combinatorial or geometric invariants of this space? \\
	\end{quest}

	Our approach is to consider the $S_n$-action on $\overline{\mathcal{M}_{0, n}}$ coming from permuting the $n$ marked points. In fact, the existence of this group action actually imposes a sort of rigidity condition that uniquely characterizes $\overline{\mathcal{M}_{0, n}}$ among other wonderful compactifications of the $A_{n - 2}$ hyperplane arrangement complement (\cite{CG1}, \cite{CG2}). In fact, the $S_n$-actions are the only automorphisms of $\overline{\mathcal{M}_{0, n}}$ if $n \ge 5$ \cite{BM}. On the other hand, there are many $S_{n - 1}$-invariant building sets \cite{GS}. In Section \ref{unqsym}, we show which differences in intersection patterns of associahedra and corresponding parallel faces of permutohedra induce the failure of the action to extend (Theorem \ref{polychar}). This implies that most of the terms of the cohomology ring (equivalently the Chow ring) except the linear terms induce this uniqueness property (Corollary \ref{cohsymef}). Further details are given below. \\

	\begin{thm} (Theorem \ref{polychar}, Corollary \ref{cohsymef})
		The failure of the extension of the action comes from pairs of intersecting faces on associahedra where parallel pairs of faces on corresponding permutohedra don't even share a face which is adjacent to each of them (Theorem \ref{polychar}). As a consequence of the construction, we show that most of the terms of the cohomology ring (equivalently the Chow ring) of degree $\ge 2$ induce this failure of extension of the group action (Corollary \ref{cohsymef}). On the other hand, the linear terms don't seem to reflect the effect of this $S_n$-action. \\
	\end{thm}

	\color{black}
	Next, we consider the connection of the $S_n$-action to positivity properties related to Hodge--Riemann relations for the Chow ring of $\overline{\mathcal{M}_{0, n}}$ in Section \ref{logchern}. Note that linear divisors are important here since these relations start with the choice of a Lefschetz element representing a nef divsior.  Although the $F$-nef divisors (which give many of the known nef divisors) aren't necessarily preserved by the $S_n$-action, they often seem to send boundary divisors that have a nonzero intersection with a fixed $F$-curve an something that has intersection product equal to 0 with the same $F$-curve (Corollary \ref{genbdact}). When we consider the subset of degree 1 elements which are invariant elements of the Picard group of $\overline{\mathcal{M}_{0, n}}$, we obtain some restrictions on possible log concave sequences which can be obtained from the degree 1 Hodge--Riemann relations. \\
	
	Roughly speaking, the $S_n$-invariance means that many of these sequences can be directly obtained Chern classes of families of globally generated vector bundles with special recursive properties. These bundles have been used to study the nef cone of $\overline{\mathcal{M}_{0, n}}$. The resulting terms terms are polynomials in generalizations of Littlewood--Richardson coefficients and are coefficients of polynomials which can be interpreted as volumes or resultants coming from linear deformations of divisors. These generate all possible log concave sequences arising from degree 1 Hodge--Riemann relations in $A^\cdot(\overline{\mathcal{M}_{0, n}})$ when $n \le 2000$ and conjecturally for all $n$ (Theorem \ref{polychar}). We give a summary of the results above with some comments related to interpretation of the terms of these sequences. \\

	\begin{thm} (Theorem \ref{polychar})
	There is a family polynomials in quantum Littlewood-Richardson coefficients with special recursive properties generating all possible log concave sequences arising from degree 1 Hodge--Riemann relations in $A^\cdot(\overline{\mathcal{M}_{0, n}})$ when $n \le 2000$ and conjectrually for all $n$ (Theorem \ref{polychar}). The coefficients in these polynomials are polynomially generated by certain initial parameters (e.g. partitions), factorials, and multinomial coefficients. In some sense, the degree 1 Hodge--Riemann relations from a specific class of sequences give a general way of thinking about volumes on $\overline{\mathcal{M}_{0, n}}$ in the $S_n$-equivariant setting. A special case gives rise to resultant polynomials with Chern roots subsituted in for the variables. \\
	\end{thm}

	Finally, the main tool used in the analysis of log concavity sequences related to $S_n$-invariance can be used to find connections between higher Hodge--Riemann relations (of a different ring) and the geometry of $\overline{\mathcal{M}_{0, n}}$ by via lower triangular Toeplitz matrices (Proposition \ref{toephrm0n}) in Section \ref{m0nhrhigh}.

	\section*{Acknowledgments}
	I would like to thank Benson Farb for the helpful comments. Also, I would like to thank Benson Farb and Karim Adiprasito for their encouragement.

	\section{ Hidden symmetries of $\overline{\mathcal{M}_{0, n}}$ arising in polytope geometry and cohomology groups } \label{unqsym}

	The $S_n$-action on $\overline{\mathcal{M}_{0, n + 1}}$ does \emph{not} extend to wonderful compactifications of the complement of the $A_{n - 2}$ hyperplane arrangement that use non-minimal building sets (\cite{CG1}, \cite{CG2}). Conversely, there are no other automorphisms of $\overline{\mathcal{M}_{0, n}}$ for $n \ge 5$ (p. 3 of \cite{CKL}). From the point of view of representations, this can come from decomposition of restrictions of $S_{n + 1}$-representations to $S_n$ vs. ones attainable from $S_n$ in general ($n = 4$ in Remark 5.4 on p. 29 of \cite{GS}). Note that there is a description of which representations of $A^k(\overline{\mathcal{M}_{0, n}})$ come from such restrictions (Proposition 5.12 on p. 27 of \cite{CKL}). That being said, the (equivariant) Poincar\'e polynomials coming from quotients of the $S_{n + 1}$-action can be generated by those for ``restricted actions'' coming from $A_{n - 1}$ and applying to any wonderful compactification of the hyperplane complement by Corollary 6.6 on p. 34 of \cite{CKL}. Looking at the spherical models of the minimal and maximal building sets \cite{CG2}, the difference comes from symmetries of the face poset of the associahedron of a given dimension that don't extend to those of the permutohedron of the same dimension. In particular, this has to do with intersection patterns of the faces.  \\

	To give more specific information on the pairs of faces involved, we look at the action of $S_{n + 1}$ on the Feichtner-Yuzvinsky generators of the cohomology ring and translate this back to the geometry of the associahedra and permutohedra. \\

	\begin{prop} \label{extrachain}
		Fix $j \in [n]$. The action of $\psi_m = (0, m) \in S_{n + 1}$ ($S_{n + 1}$ being the permutations of $\{ 0, 1, \ldots, n \})$ on disjoint chains of subsets of $[n]$ \[ S = S_1 \subset \cdots \subset S_k \] and \[ T = T_1 \subset \cdots \subset T_\ell \] with $S_k \cap T_\ell = \emptyset$ has the following effect on intersections:
		
		\begin{enumerate}
			\item If $m \notin S_k, T_\ell$, then $\psi S = S$ and $\psi T = T$ are still dijsoint. 
			
			\item Suppose that $m \in S_k$ and $m \notin T_k$. Let $S_i$ be the first element of $S$ to contain $m$. Then we end up with the two disjoint chains \[ \psi_m S_1 = S_1 \subset \cdots \subset \psi_m S_{i - 1} = S_{i - 1} \] and \[ \psi_m T_1 = T_1 \subset \cdots \subset \psi_m T_\ell  = T_\ell \subset \psi_m S_k \subset \psi_m S_{k - 1} \subset \cdots \subset \psi_m S_{i + 1} \subset \psi_m S_i. \] Note that $\psi_m A = (A^c \cup m) \setminus 0$ if $m \in A$. 
		\end{enumerate}
		
	\end{prop}
	
	\begin{proof}

		In general, we will take $S_{n + 1}$ to be the set of permutations of $\{ 0, 1, \ldots, n \}$ and write $S_n$ for the subgroup that fixes $0$ (i.e. parametrizing permutations of $\{ 1, \ldots, n \}$). Without loss of generality, we may assume that $j = 1$. One way to describe the action of $(0, j) \in S_{n + 1}$ on the cohomology of $\overline{\mathcal{M}_{0, n + 1}}$ is to say that $c_T \mapsto c_T$ if $j \notin T$ and $c_T \mapsto c_{(T^c \cup j) \setminus 0}$, where the complement is taken in $\{ 0, 1, \ldots, n \}$ and $c_J$ denotes a Feichtner-Yuzvinsky generator corresponding to $J \subset [n]$ with $3 \le |J| \le n - 2$ (p. 80 -- 81 of \cite{Gt}). The main point is to see what this does to chains of subsets $S_1 \subset \cdots \subset S_k$ with $S_i \subset [n]$ and pairs of disjoint subsets $P, Q \subset [n]$ with $P \cap Q = \emptyset$. \\
		
		For the chains of subsets, we first consider the case $k = 2$. Consider $A \subset B$ with at least one of $A$ or $B$ containing 1 since neither $A$ nor $B$ are affected by $(0, 1)$ if they don't contain 1. We split into cases where $1 \in A$ and $1 \notin A$.  If $1 \in A$, then $1  \in B $ since $A \subset B$. under the action of $\psi := (0, 1)$, we have $\psi A = (A^c \cup 1) \setminus 0$ and $\psi B = (B^c \cup 1) \setminus 0$. We then have $\psi A \supset \psi B$ since complements reverse the inclusion. Now suppose that $1 \notin A$ and $1 \in B$. Then, we have that $\psi A = A$ and $\psi B = (B^c \cup 1) \setminus 0$. This means that $\psi A \cap \psi B = \emptyset$ since $A \subset B \Rightarrow A \cap B^c = \emptyset$. The general case of chains $S_1 \subset \cdots \subset S_k$ is similar. If none of the $S_i$ contain $1$, then nothing happens to the chain. Suppose that $S_i$ is the first element of the chain to contain 1. Then, we have that $\psi S_j = S_j$ for $1 \le j \le i - 1$ and $\psi S_j = (S_j^c \cup 1) \setminus 0$ for each $i \le j \le k$. The reasoning in the $k = 2$ case implies that we have $\psi S_1 \subset \cdots \subset \psi S_i$ and $\psi S_{i + 1} \supset \cdots \supset \psi S_k$ with $\psi S_u \cap \psi S_v = \emptyset$ if $1 \le u \le i - 1$ and $i \le v \le k$.  \\
		
		For the case of disjoint sets, we start with $P, Q \subset [n]$ such that $P \cap Q = \emptyset$. If $1 \notin P, Q$, then $\psi P \cap \psi Q = \emptyset$ since $\psi P = P$ and $\psi Q = Q$. Suppose that $1 \in P$ and $1 \notin Q$. Then, we have $\psi P = (P^c \cup 1) \setminus 0$ and $\psi Q = Q$. Note that $P \cap Q = \emptyset \Rightarrow Q \subset P^c$. This implies that $\psi Q \subset \psi P$. \\
		
		We can combine the arguments above to describe how the action of $\psi = (0, 1)$ affects the intersection patterns between dijoint chains $S := S_1 \subset \cdots S_{i - 1} \subset S_i \subset \cdots S_k$ and $T := T_1 \subset \cdots T_{j - 1} \subset T_j \subset \cdots \subset T_\ell$.  By ``disjoint'', we mean that $S_k \cap T_\ell = \emptyset$. If $1 \notin S_k, T_\ell$, then we end up with a pair of disjoint chains after the action of $\psi$. Suppose that $1 \in S_k$ and $1 \notin T_\ell$. Let $S_i$ be the first element of $S$ to contain $1$. Since comparisons within chains were done above, we would like to compare $\psi S_u$ with $\psi T_v$. We can split this into cases where $1 \notin S_u, T_v$ and $1 \in S_u, 1 \notin T_v$ since the remaining case is identical to the second one. If $1 \notin S_u, T_v$ (i.e. $1 \le u \le i - 1$), then $\psi S_u = S_u$ and $\psi T_v =  T_v$. Since $S_k \cap T_\ell = \emptyset$, we also have that $S_u \cap T_v = \emptyset$ and $\psi S_u \cap \psi T_v = \emptyset$. Suppose that $1 \in S_u$ and $1 \notin T_v$ (i.e. $i \le u \le k$). Our previous arguments then imply that $\psi T_v \subset \psi S_u$. \\
	\end{proof}
	
	We can use the result above to make a statement on how the geometry of polytopes reflects the unique characterization of $\overline{\mathcal{M}_{0, n + 1}}$ by symmetries. 
	
	\begin{thm} \label{polychar}
		The failure of the $S_{n + 1}$-action on $\overline{\mathcal{M}_{0, n + 1}}$ to extend to other wonderful compactifications of the $A_{n - 1}$-arrangement using non-minimal building sets on $[n]$ come from pairs of faces on $(F_1, F_2)$ on the $(n - 2)$-dimensional associahedron $A_{n - 2}$  and $(G_1, G_2)$on the $(n - 2)$-dimensional permutohedron $P_{n - 2}$  satisfying the following properties:
		
		\begin{itemize}
			\item $F_i$ is parallel to $G_i$ for $i = 1, 2$.
			
			\item $F_1$ and $F_2$ intersect on a common face of the associahedron $A_{n - 2}$.
			
			\item $G_1$ and $G_2$ do \emph{not} intersect a common face of the permutohedron $P_{n - 2}$. 
		\end{itemize}
		
		In particular, there is a large number of pairs faces of codimension $\ge 2$ (most of them) satisfying these properties. Further details are listed near the end of the proof.
		
	\end{thm}
	
	\begin{rem}
		The building sets we're working wtih parametrize $n!$ disjoint copies of the associahedra and permutohedra discussed and the ones used in \cite{CG1} or \cite{CG2} may be different from those in Section 8 of \cite{Pos}. Note that the nested set description for associahedra given in Section 8.2 on p. 1050 \cite{Pos} is the same as the ones given in the nested sets associated to the fundamental building set from subspaces spanned by simple roots on p. 184 of \cite{CG2}. As for the permutohedral nested sets in Section 8.1 on p. 1050 of \cite{Pos}, they can be obtained using an embedding of subsets of $[n]$ into the collection of partitions of $[n]$ by sending $A \subset [n]$ to the partition consisting of $A$ and the elements of $A^c$ as singletons. This can be used to embed the minimal building set into the maximal building set associated to $A_{n - 1}$.
	\end{rem}
	
	\begin{proof}
		

		It suffices to show that $S_{n + 1}$-action does not extend in the case of the \emph{maximal} building set consisting of partitions of $[n]$ such that at least one part has size $\ge 2$ and nested sets consisting of chains of partitions of $[n]$ with respect to the refinement relation with inclusion if every subset in the first collection is contained in some subset in the second collection (p. 176 of \cite{CG2}). Equivalently, we can think of the building sets as collections of disjoint subsets of $[n]$ of size $\ge 2$ with the same definitions used for the nested sets and their refinement relations. These correspond to all the linear subspaces obtained as spans of roots in $A_{n - 1}$. \\
		
		Given the inclusion-reversing correspondence between nested sets associated to the building set and faces of generalized permutohedra (which include associahedra and permutohedra -- see Theorem 7.4 on p. 1045 and Proposition 7.5 on p. 1046 of \cite{Pos}), two faces intersect if and only if the union of the corresponding nested sets (i.e. collection of subsets of $[n]$ in either of the nested sets) is a nested set. In our case, we are looking for a minimal nested set of the maximal building set that contains the pair of nested sets of the minimal building set obtained after taking reflections by $(0, m) \in S_{n + 1}$. More concretely, this means checking whether the disjoint chains making up each of the nested sets can be interlaced to form a larger collection of disjoint chains that uses up all of the subsets of $[n]$ involved where unions of $\ge 2$ disjoint subsets of $[n]$ do \emph{not} give building sets. In the case of the maximal building sets, we only consider chains since unions of two or more disjoint subsets of $[n]$ of size $\ge 2$ are always contained in the given building set. \\

		The difference comes from the interaction between this interlacing behavior in associahedra and permutohedra under the $S_{n + 1}$-action (more specifically the transpositions $(0, m)$). By Proposition \ref{extrachain}, the set of disjoint pairs of chains of subsets of $[n]$ to disjoint pairs of chains of subsets of $[n]$ is preserved under the action of these transpositions. This is an example of how transpositions preserve nested sets of the \emph{minimal} building set defining $\overline{\mathcal{M}_{0, n + 1}}$. Although chains of subsets of $[n]$ in the minimal building set still yield chains in the maximal building set, there may be a problem since we only want to end up with one chain. In other words, there may be problems when we consider a pair of chains $S = S_1 \subset \cdots S_k$ and $T = T_1 \subset \cdots \subset T_\ell$ where a nonempty proper subcollection of $\{ S_1, \ldots, S_k \}$ contain $m$ (i.e. nonempty but not the entire collection) and none of the $T_v$ contain $m$. This explains the discrepancies in the compatibilty of the $S_{n + 1}$-action when we consider whether the union of nested sets corresponding to the intersection of two faces of an associahedron gives a face of the permutohedron. \\
		
		If we start 3 or more number of disjoint chains of subsets of $[n]$, we can still use some similar reasoning since there is only one chain that can contain $m$. This would imply that there are two or more chains not containing $m$, which remain disjoint after the action of $(0, m)$. However, disjoint chains are not permitted when considering nested sets of the maximal building set. There are even more problems whenever a proper nonempty subcollection of the sets in this special chain contains $m$. Suppose that \emph{all} the elements contain $m$. We can thus summarize the problematic cases (i.e. those forcing a pair of disjoint sets to be involved) as follows with a ``union'' of nested sets taken to mean unions with sets with elements taken to be subsets of $[n]$:

		\begin{itemize}
			\item The union of nested sets is a single chain of subsets of $[n]$ where a proper subcollection  of the subsets (i.e. some but not all of them) contains $m$.
			
			\item  The union of nested sets before applying $\psi_m = (0, m)$ is a pair of disjoint chains of subsets of $[n]$ such that neither chain contains $m$ or the one chain containing $m$ has a nonempty proper subcollection of elements containing $m$ (i.e. some but not all of the subsets of $[n]$ in the chain contain $m$).
			
			\item The union of nested sets before applying the transposition $\psi_m = (0, m)$ contains $\ge 3$ disjoint chains of subsets of $[n]$.
		\end{itemize}
		
		For any collection of disjoint chains of subsets of $[n]$ of length $\ge 2$, there is \emph{some} $m$ where one of the conditions above is satisfied. This is because there is always some $m$ which is not a common element among the subsets of $[n]$ in the chain. The comment on the dimension comes from the fact that the dimension of a face associated to a nested set $N$ is $n - |N|$, where $|N|$ is the number of subsets of $[n]$ contained in $N$ (Proposition 7.5 on p. 1046 of \cite{Pos})
		
	\end{proof}
	
	Here is some further information on polytopal interpretation of the transpositions $(0, m) \in S_{n + 1}$ considered above.

	\begin{rem} ~\\
		\vspace{-3mm}
		\begin{enumerate}
			\item The transpositions $(0, m) \in S_{n + 1}$ can be interpreted in terms of permutations of coefficients of linear forms defining faces of the polytope. Without loss of generality, set $m = 1$. In the construction of a linear function defining the face corresponding to a given nested set (see proof of Theorem 7.4 on p. 1045 -- 1046 of \cite{Pos}), sending $T \mapsto (T^c \cup 1) \setminus 0$ would result in the same partition blocks with the ones from subsets of $[n]$ in the given nested set that contain $m$ being permuted.
			
			\item The ``problematic conditions'' listed at the end of the proof is connected with the number of components of a building set in \cite{Pos}. In certain cases, there are connections with actual connected components (Example 7.2 on p. 1044 of of \cite{Pos}) and the connected components can be used to build recursive formulas for generating functions for face vectors of polytopes (Theorem 7.11 on p. 1049 of \cite{Pos}). Note that nested sets associated to building sets in \cite{Pos} are assumed to contain the maximal building sets unlike the definition used in \cite{CG1} and \cite{CG2} (Definition 7.3 on p. 1044 of \cite{Pos} vs. Definition 32 on p. 193 of \cite{CG2} or Definition 3.2 on p. 11126 of \cite{CG1}). 
			
			\item On the level of cellular approximations of the diagonal, identical nested sets also yield corresponding pairs (e.g. from the permutohedra and associahedra above) belonging to the image of the cellular diagonal (Theorem 3.16 on p. 35 of \cite{LA}). Looking at the example comparing images of diagonals of permutohedra and associahedra below the result on universal formulas of operahedra, we can consider how the incompatibility arising from differences in intersection properties of faces of polytopes translates to the diagonals and how the action of the $(0, j) \in S_{n + 1}$ could relate to enumerative properties studied recently. Alternatively, one can study M\"obius function-related invariants and action on the homology of the face posets of the polytopes involved (e.g. in relation to the Tamari and ordered partition lattices). Note that there is a decomposition involving representations induced by stabilizers of elements, which are given by cyclic groups (see p. 187 and rest of Section 4 of \cite{CG2}, Theoerem 4.4.1 on p. 78 and p. 77 -- 85 of \cite{Wa} in general). 
		\end{enumerate}
	\end{rem}
	
	Moving to cohomology, the proof of Theorem \ref{polychar} has some automatic implications for the cohomology of $\overline{\mathcal{M}_{0, n + 1}}$ (which is isomorphic to its Chow ring). 
	
	\begin{cor} \label{cohsymef}
		Any element of the Chow ring $A^\cdot(\overline{\mathcal{M}_{0, n + 1}})$ of degree $\ge 2$ uses terms inducing the failure of the natural $S_{n + 1}$-action fails to extend to maximal building set associated to $A_{n - 1}$. On the other hand, the $S_{n + 1}$-action on the elements of degree 1 seem to come from restrictions to $S_n$ (i.e. do not make use of the extended action). By ``induces'', we mean that the supports of the monomials used give examples of nested sets sent to ones not necessarily coming from the maximal building set after the $S_{n + 1}$-action. 
	\end{cor}
	
	\begin{proof}
		The degree $\ge 2$ case essentially follows from the end of the proof of Theorem \ref{polychar}. More specifically, such an issue would come up in any chain $S_1 \subset \cdots \subset S_k$ of subsets of $[n]$ of length $k \ge 2$. This already implies that most possible terms would experience this issue. The only possible instances where this might not occur come from collections of disjoint sets. However, note that $D_S = D_{S^c}$ in Keel's presentation of the cohomology ring \cite{Ke}. This is compatible with the generators of the Feichtner-Yuzvinsky basis which we have been working with since we can use the same subsets most of the time (i.e. everything except for the full subset $[n]$ -- see p. 64 -- 65 of \cite{Gt}). In particular, note that $P \cap Q = \emptyset \Longrightarrow P^c \subset Q$ and we have a chain of length 2. For the degree 1 elements, this is a comparison of Proposition 5.12 on p. 27 of \cite{CKL} and Corollary 6.2 on p. 31 of \cite{CKL}.
	\end{proof}
	
	\begin{rem} (Comparison with previous geometric considerations) \\
		\vspace{-3mm}
		\begin{enumerate}
			\item As an amusing counterpoint, the difference between a (normalized) Poincar\'e polynomial of $\overline{\mathcal{M}_{0, n + 1}}/S_{n + 1}$ and that of $\overline{\mathcal{M}_{0, n}}$ can be written entirely in terms of Poincar\'e polynomials of quotients of the form $\overline{\mathcal{M}_{0, r}}/S_{r - 1}$ for smaller $r$. In some sense, this means that the difference between the $S_{n + 1}$-action and $S_n$-action from the root system $A_{n - 1}$ can be understood in terms of the ``standard'' root system action on moduli spaces of smaller sets of points.
			
			\item The extended $S_{n + k}$-actions on codimension $k$ strata of the minimal building sets discussed in Section 10 of \cite{CG1} seem to give a combinatorial interpretation of parts of the representation decompositions in \cite{CKL}. More specifically, we can see this by comparing permutations acting on monomials and representations $\Ind_G^{S_{n + k}} Id$ for partition stabilizers $G$ on p. 11144 -- 11145 of \cite{CG1} with the decompositions $U_T$ involving stabilizers of trees $T$ on p. 27 and p. 30 of \cite{CKL}. The trivial reprsentations involved are defined on p. 28 of \cite{CKL} (which also gives a decomposition of the representation in terms of Chow rings of projective spaces). \\
		\end{enumerate}
	\end{rem}

	\section{Implications for positivity-related properties} \label{posimp}

	\subsection{ $S_n$-equivariant parts of the Chow ring of $\overline{\mathcal{M}_{0, n}}$, nef divisors, and log concavity} \label{logchern}

	Although the linear forms don't contribute to the uniqueness of the $S_{n + 1}$-action among wonderful compactifications of $A_{n - 1}$-arrangements, we can study their connections to positivty-related properties of the Chow ring. In particular, the Hard Lefschetz and Hodge--Riemann properties of the Chow ring $\overline{\mathcal{M}_{0, n + 1}}$ requires the choice of a linear form representing a nef divisor. We will consider the effect of the $S_{n + 1}$-action on the nef property and what the $S_{n + 1}$-invariant version of volumes on this space would mean. \\
	
	An issue with $\overline{\mathcal{M}_{0, n}}$ is that the nef divisors aren't fully understood (although there was a conjecture involving $F$-curves which was proven in special cases although it is false in general). Also, note that the $S_n$-equivariant case is still open (see \cite{MS}). \\
	
	We can consider whether $F$-nef divisors or interesting subsets of them are preserved under the $S_{n + 1}$-action. To do this, we need to look at the intersection with $F$-curves. We will use mostly use the notation from \cite{Pi} except that we will use the notation for the point set  $\{ 0, \ldots, n \}$ (e.g. see p. 64 of \cite{Gt}) instead of $\{ 1, \ldots, n \}$. \\
	
	\begin{defn} \label{fcurvedef} (p. 1319 of \cite{Pi}) \\
		Write $D_S = D_{S, T}$ with $T = S^c$ ($D_S = D_{S^c}$). Given a partition $A_1, A_2, A_3, A_4$ of $\{ 0, 1, \ldots, n \}$ into four nonempty subsets, the the intersection of $D_S$ with the correponding $F$-curve $C_{A_1, A_2, A_3, A_4}$ is

		\[ D_{S, T} \cdot C_{A_1, A_2, A_3, A_4} =  \begin{cases} 
			1 & \text{ if } S = A_i \cup A_j \text{ for some } i \ne j \\
			-1 & \text{ if } S \text{ or } T = A_i \text{ for some } i \\
			0 & \text{else.} 
		\end{cases}
		\]	
	\end{defn}
	
	We consider the effect of the extra transpositions $\psi_m = (0, m) \in S_{n + 1}$ on intersections with $F$-curves.
	
	\begin{prop} \label{extranef} ~\\
		\vspace{-3mm}
		We use the notation in Definition \ref{fcurvedef}.
		
		\begin{enumerate}
			\item Suppose that $D_S \cdot C_{A_1, A_2, A_3, A_4} \ne 0$. If $\{ m \}$ is \emph{not} a block $A_i$, then $D_{\psi_m S} \cdot C_{A_1, A_2, A_3, A_4} = 0$. If $A_i = \{ m \}$ for some $i$, then $D_S \cdot C_{A_1, A_2, A_3, A_4} = 1 \Longrightarrow D_{\psi_m S} \cdot C_{A_1, A_2, A_3, A_4} = -1$ and $D_S \cdot C_{A_1, A_2, A_3, A_4} = -1 \Longrightarrow D_{\psi_m S} \cdot C_{A_1, A_2, A_3, A_4} = 1$. 
			
			\item Suppose that $D_S \cdot C_{A_1, A_2, A_3, A_4} = 0$. If $D_{\psi_m S} \cdot C_{A_1, A_2, A_3, A_4} \ne 0$, $S$ is close to being a union of blocks $A_i, A_i \cup A_j$, or $A_i \cup A_j \cup A_k$ for distinct $i, j, k$. More precisely, the latter property needs to be true for $(S \setminus m) \cup 0$.
		\end{enumerate}
	\end{prop}
	
	\begin{proof}
		Fix a partition $A_1, A_2, A_3, A_4$ as above. If $m \in S$, applying $\psi_m = (0, m) \in S_{n + 1}$ would send $S \mapsto (S^c \cup m) \setminus 0$. Then, we have that $\psi_m S \cap S = \{ m \}$. Note that $(\psi_m S)^c = 0 \cup (S \setminus m)$. The intersections of sets coming from the first two cases are of the form $\emptyset, A_i \cup A_j$, $A_i$, or $A_i \cup A_j \cup A_k$ for distinct $i, j, k$. If $D_S \cdot C_{A_1, A_2, A_3, A_4} = \ne 0$ (i.e. $S = A_i \cup A_j$ for $i \ne j$ or $S = A_i$ for some $i$), we can only have $D_{\psi_m S} \cdot C_{A_1, A_2, A_3, A_4} \ne 0$ if $\{ m \}$ is one of the partition blocks $A_1, A_2, A_3, A_4$.  \\
		
		Suppose that $S = A_1 \cup A_2$ (i.e. $D_S \cdot C_{A_1, A_2, A_3, A_4} = 1$). If $A _1= \{ m \}$, then $\psi_m S = A_3 \cup A_4 \cup \{ m \} = A_1 \cup A_3 \cup A_4$. This means that $(\psi_m S)^c= A_2$ and $D_{\psi_m S} \cdot C_{A_1, A_2, A_3, A_4} = -1$. Now suppose that $S = A_i$ for some $i$ or $S = A_i \cup A_j \cup A_k$ for some distinct $i, j, k$ (i.e. $D_S \cdot C_{A_1, A_2, A_3, A_4} = -1$). We do not consider the first case since we assume that $m \in S$ and $|S| \ge 2$. So, we take $S = A_1 \cup A_2 \cup A_3$ with $A_1 = \{  m \}$. Then, we have that $\psi_m S = A_4 \cup \{ m \} = A_4 \cup A_1$ and $D_{\psi_m S} \cdot C_{A_1, A_2, A_3, A_4} = 1$. \\
		
		In general, we have that $\psi_m S$ is of the form $A_i, A_i \cup A_j$, or $A_i \cup A_j \cup A_k$ for distinct $i, j, k$ if and only if the same is true for $(S \setminus m) \cup 0$. Suppose that $D_S \cdot C_{A_1, A_2, A_3, A_4} = 0$. If $D_{\psi_m S} \cdot C_{A_1, A_2, A_3, A_4} \ne 0$, the previous observation implies that $S$ is still close to being a union of partition blocks $A_i$ even if it is not. \\
	\end{proof}
	
	In order for a divisor to be $F$-nef, we need the intersection with \emph{every} $F$-curve $C_{A_1, A_2, A_3, A_4}$ to be nonnegative. The actions described in Proposition \ref{extranef} seem a bit complicated when we take arbitrary elements of $A^1(\overline{\mathcal{M}_{0, n + 1}})$ that have nonnegative intersections with $C_{A_1, A_2, A_3, A_4}$. For example, Part 1 of Proposition \ref{extranef} seems to imply that $F$-nef divisors are \emph{not} preserved under the action of $\psi_m$. However, they imply that $\psi_m = (0, m) \in S_{n + 1}$ usually sends divisors to the boundary of the $F$-cone. \\
	
	\begin{cor} \label{genbdact}
		If $\{ m \}$ is not a block of the partition $A_1, A_2, A_3, A_4$ of $\{  0, \ldots, n \}$ and $S$ contains $\ge 2$ elements which are \emph{not} contained in $A_i, A_i \cup A_j$, or $A_i \cup A_j \cup A_k$ for any disinct $i, j, k$, then $D_{\psi_m S} \cdot C_{A_1, A_2, A_3, A_4} = 0$. 
	\end{cor}

	This ties into the considerations about suitable nef divsiors to use as Lefschetz elements (and possibly respecting the symmetric group action). As it turns out, vector bundles which have been used to study the nef cone of $\overline{\mathcal{M}_{0, n}}$ give restrictions on log concave sequences from the degree 1 Hodge-Riemann relations using $S_{n + 1}$-invariant elements of the Picard group. More specifically, the Chern classes of these vector bundles seem to account many if not all of the log concave sequences arising from the degree 1 case of the Hodge--Riemann relations (e.g. as used on p. 445 of \cite{AHK}). The term ``many'' includes $n \le 2000$ and conjectually all $\overline{\mathcal{M}_{0, n}}$. It also gives a way of expressing volumes on $\overline{\mathcal{M}_{0, n}}$ using polynomials with log concave coefficients when we impose $S_n$-equivariance. Note that the assumptions come from results on bases for $\Pic(\overline{\mathcal{M}_{0, n + 1}})^{S_{n + 1}}$ on p. 11 of \cite{GMu}. The coefficients themselves are polynomials in (quantum) Littlewood-Richardson ceofficients. It would be interesting if there are connections to representation decompositions since Poincar\'e duality holds at the level of trees involved in decompositions of the grade $k$ part of the Chow ring into representations induced by subgroups of the symmetric group stabilizing trees (Remark 5.13 on p. 28 of \cite{CKL}). \\

	\pagebreak
	
	\begin{thm} \label{qlrpolylogc} ~\\
		\vspace{-3mm}
		\begin{enumerate}
			\item Log concave sequences from (mixed) Hodge--Riemann relations inovling $S_n$-invariant elements of $\Pic \overline{\mathcal{M}_{0, n}}$ can be generated by limits of volumes of deformations of divisors divided by corresponding binomial coefficients given by polynomials in (quantum) Littlewood-Richardson coefficients. If $n \le 2000$, this accounts for all possible log concave sequences coming from the degree 1 Hodge--Riemann relations. Assuming a conjectural result, this accounts for log concave sequences for all $n$. \\

			There is a natural recursive structure among the vector bundles involved (Remark \ref{confrec}). More explicitly, the coefficients of these quantum Littlewood--Richardson coefficients are polynomials in initial parameters of these divisors (e.g. partitions), multinomial coefficients, and factorials. Also, the first Chern classes of these bundles are conjectured to be numerically equivalent to Gromov--Witten divisors discussed in \cite{BG} and \cite{CGHKLX}. \\
			
			\item After dividing each term by a corresponding binomial coefficeint, we can also interpret the sequences in Part 1 as the coefficients of a single variable polynomial given by the volumes of linear deformations of divisors in the nef cone. Applying this to cases where the vector bundles have rank $n - 3$, we obtain a family of resultants in Chern roots where these normalized coefficients form a log concave sequence.

		\end{enumerate}

	\end{thm}

	\begin{proof}
		\begin{enumerate}
			\item The log concavity comes from a result of Khovanskii--Teissier.
			
			\begin{thm} (Khovanskii--Teissier, Corollary 1.6.3(a) and Example 1.6.4 on p. 90 of \cite{Laz}) \label{kt} \\
				Let $X$ be an irreducible complete variety or scheme of dimension $n$, and let $\alpha, \beta \in N_1(X)_{\mathbb{R}}$ be nef classes on $X$. Then, the sequence $s_i = \alpha^i \cdot \beta^{n - i}$ is log concave.
			\end{thm}
			
			The relation to (quantum) Littlewood--Richardson coefficients comes from the following results:
			
			An observation connected to the key ideas in \cite{BST} and \cite{BEST} which can be used to show log concavity of coefficients of shifted characteristic polynomials of matroids is the following:
			
			\begin{cor}(Tseng \cite{Ts}) \label{ktchern}
				If $X$ is an $n$-dimensional projective variety and $V \longrightarrow X$ is a globally generated vector bundle and $\alpha$ is a nef divisor on $X$, then the sequence $s_i = \alpha^{n - i} c_i(V)$ is log concave. In other words, the degrees of the Chern classes of $V$ with respect to the nef divisor $\alpha$ form a log concave sequence.
			\end{cor}
			
			\begin{proof} (Sketch)
				This follows from considering the short exact sequence $0 \longrightarrow V' \longrightarrow \mathcal{O}_X^N \longrightarrow V \longrightarrow 0$, where the second nontrivial map is well-defined since $V$ is globally generated. Taking the second nef divisor to be the ample line divisor on $\mathbb{P}^{N - 1}$ and interpreting the $c_i(V)$ as degeneracy loci  of sections (along with the definition of the kernel as linear combinations of sections vanishing) gives the desired conclusion.
			\end{proof}

			While there are generalizations of the framework used for matroids for $\overline{\mathcal{M}_{0, n}}$ using the associated wonderful compactification, we make use of geometric properties specific to $\overline{\mathcal{M}_{0, n}}$ itself. In particular, we take the vector bundle $V$ to be a \emph{conformal block} on $\overline{\mathcal{M}_{0, n}}$, which is known to be globally generated by the following result.
			
			\begin{thm} (Fakhruddin, Lemma 2.5 on p. 6 of \cite{Fak}) \\
				All the conformal blocks $\mathbb{V}_{\overline{\lambda}}$ are generated by their global sections, therefore so are all the determinant line bundles $\mathbb{D}_{\overline{\lambda}}$. In particular, the $\mathbb{D}_{\overline{\lambda}}$ are nef line bundles.
			\end{thm}
			
			The connection to (quantum) Littlewood--Richardson coefficients comes from the fact that the ranks of conformal blocks $\mathbb{V}(\mathfrak{sl}_{r + 1}, \lambda^\cdot, \ell)$ on $\overline{\mathcal{M}_{0, n}}$ by the following result
			
			\begin{thm} (Cohomological form of Witten's Dictionary, Theorem 2.2 on p. 7 of \cite{CGHKLX}) \label{witdict} \\
				Let $\lambda^\cdot$ be a collection of $n$ partitions contained in an $r \times \ell$ rectangle satisfying $\sum_{i = 1}^n |\lambda^i| = (r + 1)(\ell + s)$ for some $s \in \mathbb{Z}$. Then the rank $R$ of $\mathbb{V}(\mathfrak{sl}_{r + 1}, \lambda^\cdot, \ell)$ on $\overline{\mathcal{M}_{0, n}}$ can be computed as follows:
				
				\begin{enumerate}
					\item If $s \le 0$, then \[ R = \int_{G(r + 1, r + 1 + \ell + s)} \sigma_{\lambda^1} \cdot \sigma_{\lambda^2} \cdot \cdots \cdot \sigma_{\lambda^n} = c_{\lambda^\cdot}^{(\ell + s)^{r + 1}}. \]
					
					\item If $s \ge 0$, then $R = c_{\lambda^\cdot, (\ell)^s}^{s, (\ell^{r + 1})}$.
				\end{enumerate}
				
			\end{thm}  
			
			The quantum Littlewood--Richardson coefficients are defined analogously to the the generalized Littlewood--Richardson coefficients from Schubert calculus except that quantum cohomology is used instead (which involves a $\mathbb{Z}[q]$-module and a degree specification). As mentioned in \cite{CGHKLX}, they are obtained from the classical products and removing rim hooks. Apart from the upper bound by the classical coefficients, we have a limiting relation indicated in (16) on p. 26 of \cite{GKOY} with a reference to the steps involving both the Littlewood--Richardson rules and the quantum version of the coefficients on p. 924 and 13.6 on p. 985 of \cite{Be}. \\

			In work of Fakhruddin \cite{Fak} and Gibney--Mukhopadhyay \cite{GMu}, it is shown that Chern classes/Chern characters of conformal blocks are polynomials $\psi$-classes, boundary divisors, (quantum) Littlewood--Richardson coefficients, and multinomial coefficients. For example, Fakhruddin \cite{Fak} obtains a formula for the degree of the first Chern class of a conformal block on any $F$-curve in Section 3.2 of \cite{Fak}. 
			
			\begin{thm}(Gibney--Mukhopadhyay, Theorem 3.1 on p. 3 of \cite{GMu}) \\
				We have \[  c_m(\mathbb{V}(\mathfrak{g}, \lambda, \ell)) = \sum_{(m_1, \ldots, m_j) \in \mathbb{Z}_{\ge 0}, m_1 + 2 m_2 + \ldots + j m_j = m } \prod_{k = 1}^j \frac{-p_k(\mathbb{V})^{m_k}}{m_k! k^{m_k}},  \]
				
				where
				
				\[ p_k(\mathbb{V}) = \sum_{m = 0}^k \sum_{k, J} \beta_J^k \psi_1^{k_1} \psi_2^{k_2} \cdots \psi_n^{k_n} D_{J_1}^{k_{n + 1}} \cdots D_{J_m}^{k_{n + m}}  \]
				
				and 
				
				\[ \beta_J^k = (-1)^{\sum_{j = 1}^m k_{n + j} } \binom{k}{k_1, \ldots, k_{n + m}} \prod_{i = 1}^n w(\lambda_i)^{k_{n + j} \rk(\mathbb{V})_\mu(\lambda_{J_j})  }.   \]
				
				with $\psi_r$ being $\psi$-classes and $D_M$ the boundary divisors along with constants $w(\mu)$ determined by partitions $\mu$.
			\end{thm}
			
			The complete notation is listed on p. 3 of \cite{GMu}, which also mentions that the first part is simply a consequence of the splitting principle. Note that $k_1 + \ldots k_{n - m} = k$ and $m_1 + 2m_2 + \ldots + jm_j = m$ for $c_m$. Taking $c_1^a c_b$ with $a + b = n$ would only deal with top degree terms from the boundary divisors and some polynomial in the ranks with integer coefficients (treating the ranks as ``variables''). The ranks determining the coefficients in the theorem above are given by (quantum) Littlewood--Richardson coefficients. Some explicit examples in Example 4.6 on p. 7 of \cite{GMu}. To clarify, we use $c_m(\mathbb{V})$ and $c_1(\mathbb{W})$ for $\mathbb{W}$ in Proposition 5.2 on p. 8 of \cite{GMu} or Example 5.4 on p. 9 of \cite{GMu}. \\

			Finally, the products $\alpha^i \beta^{n - i}$ give sums of products of monomials in the boundary coefficients (after rewriting the monomials in the $\psi$-classes using simplifications such as sliders in \cite{GGL1}). Comparing this to the expression for such terms in the therom on p. 2 of \cite{Kau} gives the factorial terms in addition to the multinomial coefficients written above. Finally, note that there are conformal blocks whose Chern classes $c_k$ are $k^{\text{th}}$ powers of the first Chern class (e.g. see Section 4 -- 5 of \cite{GMu}). It is conjectured (Question 3.3 on p. 873 -- 874) that $c_1$ of certain conformal blocks are equal to Gromov--Witten invariants. This was shown to be true in various cases in \cite{CGHKLX}. 
			
			\item 
			Given vector bundles $A$ such that $c_1(A)$ is nef and globally generated $B$ on $\overline{\mathcal{M}_{0, n}}$ (or some other smooth projective variety), we considered sequences of the form $c_1(A)^{n - i - 3} c_i(B)$ (where $n -3 = \dim \overline{\mathcal{M}_{0, n}})$. We mainly considered the case where $A$ and $B$ are given by conformal blocks. Note that tensor products of conformal blocks coming from the same $\mathfrak{g}$ are still given by conformal blocks (p.  6 of \cite{GMu}). These numbers can be treated as coefficients of a generating function of a ``twisting polynomial'' that sort of looks like how volume polynomials are defined. \\
			
			We make use of some standard results from \cite{EH}. Note that $c_1(\mathcal{L}^{\otimes k}) = k c_1(\mathcal{L})$ if $\mathcal{L}$ is a line bundle. This implies that
			
			\begin{align*}
				c_k(\mathcal{E} \otimes \mathcal{L}^{\otimes t} )&= \sum_{\ell = 0}^k \binom{r - \ell}{k - \ell} (t c_1(\mathcal{L}))^{k - \ell} c_\ell(\mathcal{E}) \\
				&= \sum_{\ell = 0}^k \binom{r - \ell}{k - \ell} t^{k - \ell} c_1(\mathcal{L})^{k - \ell} c_\ell(\mathcal{E}). 
			\end{align*}

			Dividing out by the binomial terms, we have obtain a polynomial whose coefficients are log concave. Note that we can still obtain a constant even if $k < \rk \mathcal{E}$. In the context of conformal blocks, this applies with $\mathcal{E} = \mathcal{V}(\mathfrak{g}, \nu_1, \ell_1)$ and $\mathcal{L} = \mathbb{V}(\mathfrak{g}, \mu_1, m_1)$ if the latter is a line bundle. The left hand side would then compute $c_i(\mathbb{V}(\mathfrak{g}, \nu + t\mu_1, \ell_1 + tm_1)$. Note that there are examples in terms of conformal blocks where the left hand side has rank smaller than $k$ (which is needed for this to be nontrivial). When $k = \rk \mathcal{E}$, the left hand side is actually a sort of resultant of $\mathcal{E}$ and $\mathcal{L}^{\otimes t}$ after taking the dual of one of the two bundles by Proposition 12.1 on p. 428 of \cite{EH}). This can be taken with respect to the Chern roots or something more like a ``usual'' resultant where the variable is the index ($i$ in $c_i$). \\
			
			As for volumes, let $A$ and $B$ be nef line bundles on a smooth projective variety $X$ (e.g. conformal blocks on $\overline{\mathcal{M}_{0, n}}$) with $B$ globally generated, Note that \[ \sum_{i = 0}^d  \binom{d}{i}c_1(A)^{d - i} c_1(B)^i t^i = (c_1(A) + c_1(B)t)^d = \deg(c_1(A) + c_1(B)t).  \] If $X$ were a projective toric variety, this would be $d!$ multiplied by the volume of a polytope (e.g. see p, 277 of \cite{Kav} which includes projective spherical varieties and Theorem 4.1 on p. 278 of \cite{Kav} on the subring of the cohomology ring gnerated by Chern classes of line bundles). When we consider $S_n$-equivariant elements of the Picard group, the Chern classes of line bundles are conjectured to generate $\Pic \overline{\mathcal{M}_{0, n}}^{S_n}$ and this is known for $n \le 2000$ (p. 11 of \cite{GMu}) Note that this is \emph{not} the case for $A^1(\overline{\mathcal{M}_{0, n}})^{S_n}$.

		\end{enumerate}
		
	\end{proof}
	
	Here are some examples for small $n$.
	
	\begin{exmp}
		 If we take $n = 4$ and consider the intersection with a generic hyperplane (or whatever very ample or nef line bundle is used to define the degree in \cite{Fak} and \cite{GMu}), Corollary 2.7 on p. 7 and Corollary 3.5 on p. 14 of \cite{Fak} imply that substituting in $\mathbb{V}_1 + t \mathbb{V}_2$ for conformal blocks $\mathbb{V}_1$ and $\mathbb{V}_2$ gives a polynomial in $t$ that has polynomials in ranks of conformal blocks using Theorem 3.1 on p. 3 of \cite{GMu}. For example, suppose that $n = 5$. Then, we can use the following to compute the polynomial in terms of the ranks explicitly:
		
		By equation (6) on p. 5 of \cite{GMu}, we have
		
		\[ c_1(\mathbb{V}) = \rk(\mathbb{V})  \sum_{i = 1}^5 \sum_{i = 1}^n w(\lambda_i) \psi_i - \sum_{I \subset [n], 1 \notin I, \mu_e \in \mathcal{P}_\ell(\mathfrak{g}) w(\mu_e)} \rk(\mathbb{V}_{\mu_e}(\lambda_l)) \rk(\mathbb{V}_{\mu_e^*}(\lambda_{I^c})) D_I.  \]

		Explicit computations in the Chow ring using $\psi_i = \sum_{i, m \in I, j, k \notin I} D_I$ for fixed $j, k \ne i$ (e.g. see Example 4.8 on p. 7 -- 8 of \cite{GMu}) imply relations such as those listed below:
		
		\begin{itemize}
			\item $\psi_i^2 = 1$ for $1 \le i \le 5$, $\psi_i \cdot \psi_j = 1$ for $1 \le i \ne j \le 4$ (can take a special case of top degree monomials in $\psi$-classes being equal to multinomial coefficients determined by the dimension and the exponents)
			
			\item $\psi_5 \cdot \psi_i = 2$ for $1 \le i \le 4$
			
			\item $\psi_5 \cdot D_{i, j} = 1$ for $i \ne j$ and $i, j \ne 5$
			
			\item $D_{i, j}^2 = -1$ if $i, j \ne 5$
			
			\item $D_{i, 5}^2 = -1$ if $i \ne 5$
			
			\item $D_{1, i} \cdot D_{j, 5} = 1$ if $i \ne j$ and $i, j \ne 1, 5$ and $D_{1, i} \cdot D_{i, 5} = 0$ for $i \ne 1, 5$
			
			\item $D_{1, 5} \cdot D_{i, j} = 1$ if $i, j \ne 1, 5$

		\end{itemize}

	\end{exmp}
	
	\begin{rem} \label{confrec}
		The sequences involved also have a \emph{recursive} structure coming from the Chern classes of these bundles. As a comparison, note that $c_k = \sum_{i \le k} a_i b_{k - i}$ forms a log concave sequence if $(a_i)$ and $(b_j)$ are finite log concave sequences. For conformal blocks in general, the forgetful morphisms $f_i  : \overline{\mathcal{M}_{0, n}} \longrightarrow \overline{\mathcal{M}_{0, n - 1}}$ and gluing morphisms $\gamma : \overline{\mathcal{M}_{0, n_1 + 1}} \times \overline{\mathcal{M}_{0, n_2 + 1}} \longrightarrow \overline{\mathcal{M}_{0, n_1 + n_2}}$ and $\gamma_2 : \overline{\mathcal{M}_{0,n + 2}} \longrightarrow \overline{\mathcal{M}_{1, n}}$  yield expressions for the pullback of the conformal blocks of these morphisms in terms of direct sums of tensor products or direct sums of the original conformal blocks on a smaller number of points given by \[  \mathbb{V}_{0, n, \overline{\lambda}} \cong f_i^*(\mathbb{V}_{g, n, \overline{\lambda}'}) \] for $\overline{\lambda}$ with $\lambda_i = 0$ and $\overline{\lambda}'$ obtained by removing 0, \[ \gamma^*(\mathbb{V}_{g_1 + g_2, n_1 + n_2, \overline{\lambda}}) \cong \bigoplus_{\mu \in P_\ell} \mathbb{V}_{g_1, n_1 + 1, \overline{\lambda}^1 \mu} \otimes \mathbb{V}_{g_2, n_1 + 1, \overline{\lambda}^2 \mu } \] and \[ \gamma_2^*(\mathbb{V}_{g, n \overline{\lambda}}) \cong \bigoplus_{\mu \in P_\ell} \mathbb{V}_{g - 1, n + 2, \overline{\lambda \mu}} \] respectively by Proposition 2.4 on p. 5 -- 6 of \cite{Fak}. Note that all the bundles involved are globally generated. For terms of the form $c_1(A)$ for some conformal block $A$, tensor products are given by linear combinations of $c_1$  of the original vector bundles with coefficients given by the ranks of the bundles (which are the (quantum) Littlewood-Richardson coefficients mentioned above). For conformal blocks satsifying a ``critical condition'' (with respect to being a nef divisor), the Chern classes are of the form $c_k(\mathbb{V})= c_1(\mathbb{V})^k$. In particular, this applies to the generating set of $\Pic \overline{\mathcal{M}_{0, n}}$ mentioned above. In terms of Chern classes, the resulting recursions are (alternating) sums of products of Chern classes smaller degree of conromal blocks formed from ``smaller'' initial parameters. 
	\end{rem}
	
	\subsection{Higher Hodge--Riemann relations and the geometry of $\overline{\mathcal{M}_{0, n}}$} \label{m0nhrhigh}

	Returning to the conformal blocks used to analyze $S_n$-equivariance and log concave sequences, we find that higher Hodge--Riemann relations can be connected to the geometry of $\overline{\mathcal{M}_{0, n}}$. Although quantum cohomology often fails to satisfy various functorial properties, this is what we will use to relate the Hodge--Riemann properties (of some Artinian Gorenstein algebra) to that of $\overline{\mathcal{M}_{0, n}}$. Given a Toeplitz matrix (i.e. one with same entries in each diagonal -- see Definition 4.3 on p. 12 of \cite{MMMSW}), whether it is totally positive can be determined by whether a certain algebra satisfies the Hodge--Riemann property (Theorem 4.19 on p. 21 of \cite{MMMSW}). For the subset of Toeplitz matrices which are lower triangulat, we can relate this to the quantum cohomology of (partial) flag varieties (e.g. Grassmannians). In particular, the functions on the space of lower diagonal Toeplitz matrices (possibly satisfying some properties as in Section 8 of \cite{Rie}) are determined by the quantum cohomology of a Grassmannian (via quantum parameters involvedin partial flag varieties generalizing work of Konstant on full flag varieties). \\
	
	Throughout \cite{Rie}, a major theme is an isomorphism between coordinate rings of strata of Bruhat decompositions of certain varieties with the quantum cohomology of a quotient of by a parabolic subgroup generalizing results of Peterson (Theorem 4.2 on p. 374 of \cite{Rie}). Note that the total positivity of the Toeplitz matrices is equivalent to the positivity of (quantum) Schubert classes on appropriate parameters (part 3 of Theorem 7.2 on p. 380 of \cite{Rie}). The connection back to the geometry of $\overline{\mathcal{M}_{0, n}}$ comes from the relation between conformal blocks on $\overline{\mathcal{M}_{0, n}}$ and Gromov-Witten invariants and quantum cohomology on the Grassmannian (e.g. cohomological form of Witten's dictionary connecting ranks of conformal blocks on $\overline{\mathcal{M}_{0, n}}$ and the coefficients of a basis of quantum Schubert classes on Grassmannians in Theorem 2.2 on p. 8 of \cite{CGHKLX}). Taking the structure constants on multiplications of Schubert classes to be ranks of conformal blocks under suitable conditions on the partitions involved (see Theorem 2.2 on p. 8 of \cite{CGHKLX}), we can state the following \\

	\begin{prop} \label{toephrm0n}
		For lower triangular Toeplitz matrices, Hodge--Riemann relations on Artinian Gorenstein algebras varying with the matrices determine totally/Schubert positive parts of the variety of such Toeplitz matrices. Connecting this back to the quantum cohomology of the Grassmannian, there is a basis of functions determining coefficients which come from ranks of conformal blocks on $\overline{\mathcal{M}_{0, n}}$. \\
	\end{prop}

	Since the Schubert classes referred to above are hook length formulas, we have shown that a nonnegativity of a certain family of hook length formulas depends on higher Hodge--Riemann relations of auxiliary Artinian Gorenstein algebras. Apart from this, we can also look at relations between log concave sequences. For example, this can make use of recursive relations bewteeen involving Chern classes of conformal  blocks (e.g. Proposition 4.4 on p. 7 of \cite{GM1}). Taking $c_i(V)$ to lie in the Chow ring $A^\cdot(\overline{\mathcal{M}_{0, n}})$, we can consider how to translate these recursions in terms of the geometry of the fan determined by the simplicial complex $\Delta$ parametrizing phylogenetic trees (where facets correspond to trivalent trees and smaller subfaces are obtained by contractions of internal edges) and how these are conected to Hodge-Riemann relations and earlier computations we made or combinatorial work related to $\psi$ and $\omega$-divisors on $\overline{\mathcal{M}_{0, n}}$ (e.g. \cite{GGL1}). \\

	\color{black}

\end{document}